\documentclass[journal,10pt,letterpaper]{IEEEtran}

\usepackage{amsmath,amssymb,amsfonts,amsthm}
\usepackage{mathtools}
\usepackage{bm}
\usepackage{cite}
\usepackage{graphicx}
\usepackage{booktabs}
\usepackage{url}
\usepackage{xcolor}
\usepackage{microtype}

\newtheoremstyle{ieeeplain}
  {3pt}        % Space above
  {3pt}        % Space below
  {\itshape}   % Body font
  {}           % Indentation
  {\bfseries}  % Head font
  {:}          % Punctuation after theorem head
  {0.5em}      % Space after theorem head
  {\thmname{#1}\thmnumber{ #2}\thmnote{ #3}}

\newtheoremstyle{ieeenormal}
  {3pt}        % Space above
  {3pt}        % Space below
  {\normalfont}% Body font
  {}           % Indentation
  {\bfseries}  % Head font
  {:}          % Punctuation after theorem head
  {0.5em}      % Space after theorem head
  {\thmname{#1}\thmnumber{ #2}\thmnote{ #3}}

\theoremstyle{ieeeplain}
\newtheorem{theorem}{Theorem}
\newtheorem{corollary}{Corollary}

\theoremstyle{ieeenormal}
\newtheorem{definition}{Definition}
\newtheorem{remark}{Remark}
\newtheorem{example}{Example}
\newtheorem{assumption}{Assumption}

\newif\ifincludeauthorbios
\includeauthorbiostrue

\begin{document}

\title{Sufficient Conditions for Exact Convergence Time \\in Prescribed-Time Control}

% TODO: Replace the two placeholder authors and verify IEEE memberships.
\author{Jiyuan~Kuang,
        Second~Author,
        Third~Author%
\thanks{This work was supported in part by the National Natural Science
Foundation of China under Grants 62503132, 62373127, and 62188101; in
part by the Natural Science Foundation of Shandong Province under Grant
ZR2025QC1515; and in part by the State Key Laboratory of Alternate
Electrical Power System with Renewable Energy Sources under Grant
LAPS23005. (Corresponding author: .)}%
\thanks{Jiyuan Kuang and [coauthors at Weihai] are with the Department of
Control Science and Engineering, Harbin Institute of Technology, Weihai
264209, China (e-mail: jiyuankuang@hit.edu.cn).}%
\thanks{, and [coauthors at Harbin] are with the
Department of Control Science and Engineering, Harbin Institute of
Technology, Harbin 150001, China (e-mail: yabingao@hit.edu.cn;
jx.liu@hit.edu.cn).}%
}

\markboth{IEEE Transactions on ~,~Vol.~XX, No.~XX, 2026}%
{Kuang \MakeLowercase{\textit{et al.}}: Sufficient Conditions for Exact Convergence Time in Prescribed-Time Control}

\maketitle

\begin{abstract}
In this paper, we propose a series of theoretical results for prescribed-time control to rigorously prove
that the closed-loop system will not reach the equilibrium point before the prescribed time. Thus, it is convenient to prove that the settling time is exactly equal to the prescribed time instant. 
In other words, conservative estimation of the settling time, which often accompanies classical finite-/fixed-time or prescribed-time control results, is theoretically avoided in this paper. 
Specifically, we introduce a local Lipschitz condition, a local homogeneous linear-growth condition, and an Osgood condition
in sequence. Both reduction to absurdity and the Lyapunov method are applied to demonstrate that the system cannot converge to zero before the prescribed time instant.
Finally, the proposed framework is applied to prescribed-time controllers based on adaptive backstepping, parametric Lyapunov equations, and distributed leader-following consensus, 
 demonstrating the generality of the proposed results.
\end{abstract}


\section{Introduction}
\label{sec:introduction}

\IEEEPARstart{U}{nlike}  classical asymptotic control, the key property of finite-time control is that the convergence is achieved within a finite time interval, not asymptotically as time approaches infinity \cite{2000finite}. 
For some finite-time control, the settling time (the first time closed-loop systems converge to and then stay at the equilibrium) may be uniformly bounded for all initial conditions \cite{finite}. 
This property of finite-time stabilized systems is referred to as fixed-time stability, since the settling time is strictly proven not to exceed a value independent of initial conditions, denoted by $T$.
Conventionally, $T$ in fixed-time control is a function of the controller parameters, and conservative estimation of the settling time is inevitable \cite{2011Fixed}.

Subsequently, the prescribed-time control for high-order nonlinear systems was proposed to achieve fixed-time stability using time-dependent regulators \cite{2017prescribed1}.
It should be noted that $T$ is a constant in the prescribed-time control, independent of the initial state or any other design parameter \cite{2025dis}. 
Furthermore, the prescribed-time control has the potential to make the settling time equal to the prescribed time $T$, a constant prescribed by the controller \cite{yezongshu}.

In recent years, prescribed-time control has been widely studied, and numerous advanced results have been reported for various systems. A dynamic high-gain design was developed for prescribed-time regulation of nonlinear systems \cite{krishnamurthy}. Output-feedback prescribed-time stabilization was achieved for reaction--diffusion PDEs with boundary input delay \cite{10855508}. A sensor-delay-compensated prescribed-time observer was proposed for linear time-invariant systems \cite{espitiasteeves}. Adaptive nonlinear prescribed-time control with filterless least squares was developed for online parameter estimation \cite{2025v}. Tunable-bound protocols using hyperbolic tangent feedback were designed for prescribed-time distributed optimization \cite{zhou2026tunable}. Robust prescribed-time output feedback and its digital implementation were investigated for linear systems with disturbance \cite{orlov2026output}.

In addition to the main research directions discussed above, some studies have consistently sought to rigorously prove that prescribed-time systems cannot converge to the equilibrium point before the prescribed time. 
This is because an exact convergence time is useful for certain practical missions. For example, a piano performance task requires a robotic manipulator to strike the key at an exact moment.
In addition, there are many multi-agent missions that depend greatly on how
and when the  multi-agent system  states converge  \cite{smca}, e.g., spacecraft formation flying (PROBA-3, EXO-S, SWIFT \cite{daodan}), and  cooperative missile attacks \cite{ldytac}.

This paper focuses on exact prescribed-time convergence, i.e., the closed-loop system converges to the origin exactly at the prescribed time instant. For high-order integrator chains, backstepping prescribed-time controllers were developed, while premature convergence was rigorously ruled out through a contradiction argument and a two-sided Lyapunov analysis \cite{mytc,jas}. This Lyapunov-based framework was subsequently extended to second-order systems subject to unmatched uncertainties \cite{dianlu2}.
An alternative line of research achieves exact prescribed-time stabilization through periodic delayed feedback, thereby avoiding the terminally singular high gains commonly used in prescribed-time control \cite{strongpres}. In this framework, trajectory analysis was employed to verify the prescribed settling time. Building on the same idea, periodic delayed feedback was further integrated with a super-twisting algorithm to guarantee exact prescribed-time convergence in the presence of matched and mismatched disturbances \cite{ding2026super}.
Exact convergence at a prescribed constant settling time was also achieved using homogeneous static feedback, under which every nonzero trajectory reaches the origin precisely at the assigned time \cite{kirsticreview}. This result was further generalized to systems with input delays.
More recently, exact $T$-finite-time stability was formally characterized through a two-sided Lyapunov criterion that guarantees convergence at $T$ while excluding convergence before $T$. This result can also be applied to MIMO nonlinear stabilization and observer design \cite{Zhou2026Exact}.

For MASs, distributed control protocols and observers are generally designed according to the communication graph \cite{deHeij2025Distributed}. By introducing appropriate coordination errors, the closed-loop network dynamics can be reduced to the corresponding disagreement dynamics  \cite{2022tii}. Therefore, applying the prescribed-time control of these error dynamics allows the preceding results to be extended to multi-agent coordination problems, including  containment control \cite{19wang}, formation tracking \cite{Li2025Formation}, and leader-follower consensus \cite{Ye2025Decentralized}.
Therefore, it is easy to come up with the idea of extending $T$-finite-time stability to some time-critical multi-agent consensus problems.

This article investigates exact 
$T$-finite-time stability for a general class of prescribed-time closed-loop systems. The central idea is to separate exact prescribed-time convergence into two complementary properties: convergence to the equilibrium no later than 
$T$, which can be established using existing prescribed-time stability methods, and non-reachability of the equilibrium over the interval $[0,T)$. Accordingly, this article focuses on identifying verifiable conditions that exclude premature convergence. Three sufficient criteria with progressively weaker regularity requirements are developed, namely local Lipschitz continuity, local homogeneous linear growth, and an Osgood-type growth condition. Their logical relationships and applicability are examined through theoretical analysis and illustrative examples. The resulting framework is subsequently extended to multi-agent disagreement dynamics and applied to several representative prescribed-time control systems to rigorously certify that their settling times are exactly equal to the prescribed time  $T$.

The remainder of this article is organized as follows. Section II introduces the relevant stability notions, formulates the exact prescribed-time convergence problem, and presents motivating examples that distinguish prescribed-time convergence from exact convergence. Section III establishes the three sufficient conditions for excluding premature convergence, clarifies their hierarchical relationships, and extends the Osgood-based criterion to multi-agent disagreement systems. Section IV applies the proposed results to prescribed-time stabilization based on adaptive backstepping and parametric Lyapunov equations, as well as to distributed prescribed-time leader-following consensus. Section V concludes the article.

\section{Problem statement and preliminaries}
\label{section2}

\subsection{Preliminaries}

Consider the following system:
\begin{equation}
	\dot x=g(x,u),
\end{equation}	
where $x=[x_1,x_2,...,x_n] \in \mathbb{R}^{n}$ is the state vector, and $u(x, t) \in \mathbb{R}^{m}$ is a prescribed-time controller. Throughout this paper, $t_0=0$, and the initial value is denoted as $x(0)$.
Then, we have the following closed-loop system:
\begin{equation}
	\dot x=f(x,t),\label{ft1}
\end{equation}	
where $f:\mathbb{R}^{n}\times \mathbb{R}_{\geq 0} \rightarrow \mathbb{R}^n$ is
piecewise continuous in $t$ and locally Lipschitz in $x$.

\begin{definition}  [\cite{slotine1991applied}]
The equilibrium point \(x=\mathbf{0}\) of system \eqref{ft1} is classified as follows:
\begin{enumerate}
    \item Stable, if there exists a class \(\mathcal{K}\)
    function \(\beta\) such that
    \[
        \lVert x(t)\rVert
        \leq
        \beta\bigl(\lVert x(0)\rVert\bigr).
    \]

    \item Asymptotically stable, if it is stable and
    \[
        x(t)\rightarrow 0,
        \quad \text{as } t\rightarrow\infty.
    \]

    \item Exponentially stable, if there exist two positive
    numbers \(\lambda_{1}\) and \(\lambda_{2}\) such that, for
    sufficiently small \(x(0)\),
    \[
        \lVert x(t)\rVert
        \leq
        \lambda_{1}\lVert x(0)\rVert e^{-\lambda_{2}t},
       \quad \forall\,t\geq 0.
    \]
\end{enumerate}
\end{definition}

\begin{definition}  [\cite{Holloway2019Prescribed}]
The equilibrium point  \(x=\mathbf{0}\)  of system \eqref{ft1} is prescribed-time globally uniformly asymptotically stable in
time \(T\) (or simply, prescribed-time stable) if there exists a function
    $\mu : [0,T) \rightarrow \mathbb{R}_{+}$
with \(\mu\) increasing to \(\infty\) as \(t \rightarrow T\), and a
class \(\mathcal{KL}\) function \(\beta\) such that $ \lVert x(t) \rVert
    \leq \beta\bigl(\lVert x(0) \rVert,\mu(t)\bigr), \forall\, t \in [0,T),$
where \(T\) is a finite designed  number.
\end{definition}

The prescribed-time globally uniformly asymptotically stability is defined with a time-varying gain \(\mu\). Its main property is $\lim_{t\to T} x(t)=0$. Notice that fixed-time stability also achieves $\lim_{t\to T} x(t)=0$ \cite{Zuo2026FixedTimeConsensus},  we introduce the following equivalent definition.

\begin{definition} [\cite{zhou2026finite}]
The equilibrium point  \(x=\mathbf{0}\)  of system \eqref{ft1} is $T$-finite-time stable if it is stable, and there exists a constant $T>0$ such that 
$\lim_{t\to T} x(t)=0$.
\end{definition}

Notice that $\lim_{t\to T} x(t)=0$ does not mean that $x(t)$ converges to zero exactly at $t=T$. We introduce the following definitions to show the difference.

\begin{definition} [\cite{zhou2026finite}] \label{exactT}
The equilibrium point  \(x=\mathbf{0}\)  of system  \eqref{ft1} is exact $T$-finite-time stable if it is $T$-finite-time stable  and $x(t)\neq 0, \forall\, t\in[0,T)$ for every initial condition
\(x(0)\neq 0\). 
\end{definition}

\begin{definition} [\cite{JimenezRodriguez2020Lyapunov}]
If the equilibrium point  \(x=\mathbf{0}\)  of  system \eqref{ft1} is $T$-finite-time stable and $ x(0) \ne 0$, then the settling time is defined as $T_s := \inf\{t_s > 0 : x(t) = 0, \forall t\ge t_s\} $.
\end{definition}

The settling time is the first convergence time instant after which the trajectory stays at the origin.
 If the equilibrium point  \(x=\mathbf{0}\) of system  \eqref{ft1} is exact $T$-finite-time stable, then $T_s=T$.

\subsection{Problem statement}

To clarify the misunderstanding that the prescribed-time globally uniformly asymptotically stable system
 necessarily converges to the
origin exactly at the prescribed time,
 the following two examples are designed.

\begin{example}
Consider the scalar prescribed-time system with an arbitrarily small constant perturbation \(\varepsilon>0\):
\begin{equation*}
    \dot{x}(t)
    =
    -\frac{x(t)}{T-t}-\varepsilon,
    \quad 0\leq t<T.
\end{equation*}
Before the trajectory first reaches the origin, its solution is
\begin{equation*}
    x(t)
    =
    (T-t)
    \left[
        \frac{x_0}{T}
        +
        \varepsilon
        \ln\left(\frac{T-t}{T}\right)
    \right].
\end{equation*}
The first hitting time of the origin is
\begin{equation*}
    t_s
    =
    T\left(
        1-\exp\left(-\frac{x_0}{\varepsilon T}\right)
    \right).
\end{equation*}
An  arbitrarily small perturbation that does not vanish at the
origin causes the trajectory to reach the origin strictly before the
prescribed time.

\end{example}

Therefore,
systems subject to persistent perturbations are
not considered in this paper. The condition \(f(t,0)=0\) is imposed to
guarantee the invariance of the origin and to make the exact
settling-time analysis meaningful.

\begin{example}
Consider the following second-order system:
\begin{equation*}
\begin{aligned}
\dot x_1
&=
-\frac{1}{T-t}
\left(
\frac{1}{2}x_1
+
|x_1|^{1/2}\operatorname{sgn}(x_1)
\right)
+x_2,\\[1mm]
\dot x_2
&=
-\frac{1}{T-t}
\left(
\frac{1}{2}x_2
+
|x_2|^{1/2}\operatorname{sgn}(x_2)
\right)
-x_1,
\end{aligned}
\end{equation*}
where \(T>0\) and \(0\le t<T\). At the origin, the right-hand side is defined as zero.
Choose the Lyapunov function
$V(x)=\frac{1}{2}(x_1^2+x_2^2)$. Then
\begin{equation*}
    \dot V=
-\frac{V}{T-t}
-\frac{1}{T-t}
\left(
|x_1|^{3/2}+|x_2|^{3/2}
\right).
\end{equation*}
Notice that
\begin{equation*}
    |x_1|^{3/2}+|x_2|^{3/2}
    \ge
    (x_1^2+x_2^2)^{3/4}
    =
    (2V)^{3/4}.
\end{equation*}
We obtain
\begin{equation*}
    \dot V
    \le
    -\frac{V}{T-t}
    -
    \frac{(2V)^{3/4}}{T-t}.
\end{equation*}
Let $Z(t)=V^{1/4}(t)$. Then we have
\begin{equation*}
    \dot Z=
    \frac{1}{4}V^{-3/4}\dot V \notag\\
   \le
    -\frac{Z+2^{3/4}}{4(T-t)}\le
    -\frac{Z}{4(T-t)}.
\end{equation*}
It is easy to obtain
\begin{equation*}
x(t) \le \beta(x(0), \mu(t)), \quad \forall t \in
[0, T),
\end{equation*}
where $\beta$ is a $\mathcal{KL}$ function. Therefore, $x(t)$ is prescribed-time globally uniformly
asymptotically stable\cite {Holloway2019Prescribed}.
One can obtain
\begin{equation*}
    Z(t)+2^{3/4}
    \le
    \bigl(Z(0)+2^{3/4}\bigr)
    \left(
    \frac{T-t}{T}
    \right)^{1/4}.
\end{equation*}
Hence, \(Z(t)\) reaches zero no later than
\begin{equation*}
    t_s
    =
    T
    \left[
    1-
    \left(
    \frac{2^{3/4}}{Z(0)+2^{3/4}}
    \right)^4
    \right].
\end{equation*}
Since $Z(0)>0$, for all $x(0)\neq0$, we have $t_s<T$, i.e., $x(t)$ reaches the origin before the prescribed time \(T\).
\end{example}

Notice that prescribed-time control usually introduces a time-varying gain. We
consider the following system,
\begin{equation}  	
\dot{x}(t)
    =
    f(t,x)
    :=P_0(t,x)+
    \sum_{k=1}^{N}
    \frac{P_k(x)}{\Phi_k(T-t)}, 0\leq t<T,
 \label{mainsystem}
\end{equation}
where \(x\in\mathbb{R}^n\). The right-hand side contains $P_0(t,x)$, state-dependent numerators $P_k(x), k=1,\ldots, N$, and
time-varying denominators $\Phi_k(T-t)$. In addition, $\dot{x}(t)=\mathbf{0},\forall t \ge T$.

\begin{assumption}
\label{ass:caratheodory}
Consider system \eqref{mainsystem}.
The following conditions hold for every  $ \bar T<T$:
\begin{enumerate}
\item
For each fixed $x\in\mathbb{R}^n$, the mapping $ P_0 : t\mapsto P_0(t,x)$
is measurable on $[0,\bar T]$.
For almost every $t\in[0,\bar T]$, the mapping
$x\mapsto P_0(t,x)$
is piecewise continuous.

\item
Each mapping $P_k:\mathbb{R}^n\to\mathbb{R}^n$
is continuous.
For every compact set $K\subset\mathbb{R}^n$, there exists an integrable function $m_{\bar T,K}\in L^1([0,\bar T])$,
such that $\|P_0(t,x)\|\leq m_{\bar T,K}(t)$, $\forall x\in K$ and for almost every $t\in[0,\bar T]$.
\end{enumerate}
\end{assumption}

In some literature on prescribed-time control, $\frac{P_k(x)}{\Phi_k(T-t)}$ is $\frac {x}{T-t}$ \cite{duibiao}. However, $\frac{P_k(x)}{\Phi_k(T-t)}$ can also be other forms, e.g., $\frac{x}{\mathrm{e}^{T-t}-1}$ \cite{zhou2020finite}, and $\frac{\mathrm{e}^{ax}-\mathrm{e}^{-bx}}{T-t}$ \cite{ye2022prescribed}. 
$P_0(t,x)$ is the cross term of state $x$ or other terms.

\begin{assumption} \label{prescribe}
The equilibrium point  \(x=\mathbf{0}\)  of  closed-loop  system \eqref{mainsystem} is $T$-finite-time stable. Therefore, $x(t)$ is bounded and $\lim_{t\to T^-}x(t)=0$.
\end{assumption}

\begin{remark}
Assumption~\ref{ass:caratheodory} and \ref{prescribe} is assumed to hold throughout this paper.
Specifically, Assumption \ref{ass:caratheodory} guaranties the vector field
$f(t,x)$ is a Carath\'eodory vector field on every compact
subinterval of $[0,T)$. Therefore, for each initial condition, the
system \eqref{mainsystem} admits at least one
locally absolutely continuous solution. 
Assumption \ref{prescribe} guaranties $\lim_{t\to T^-}x(t)=0$, i.e., the equilibrium point  \(x=\mathbf{0}\)  of system \eqref{mainsystem} is $T$-finite-time stable.
\textbf{Hence, the sole objective of this paper is to identify additional conditions under which system~\eqref{mainsystem} is exact $T$-finite-time stable.}
\end{remark}

\section{Main results}

\textbf{Core idea:}
By continuity, any trajectory that reaches the origin at some time $t_s<T$ must enter and remain in a sufficiently small neighborhood of the origin over a nontrivial interval immediately preceding $t_s$. Therefore, to rule out premature convergence, it suffices to impose a local condition ensuring that no nonzero trajectory in this neighborhood can reach the origin in finite time before $T$. This observation explains why the sufficient conditions developed below are required only locally around the origin.

\subsection{Local Lipschitz continuity condition}

\begin{theorem}
\label{maintheorem}
The solution of system \eqref{mainsystem} with any nonzero initial
condition $x(0)\ne 0$
cannot reach the origin at any time \(t_s<T\),
if the following conditions hold:
\begin{enumerate}
    \item \(P_0(t,0)=0\), \( \mapsto P_0(t,x):\mathbb{R}_{\geq 0} \times \mathbb{R}^n\to\mathbb{R}^n\)
    is locally Lipschitz in $x$ and piecewise continuous in $t$;

    \item\(P_k(0)=0\),  \( \mapsto P_k(x):\mathbb{R}^n\to\mathbb{R}^n\)
    is locally Lipschitz in $x$, \(k=1,\ldots,N\);

    \item
    $\Phi_k(T-t):[0,T)\to\mathbb{R}_{> 0}$ is continuous, \(k=1,\ldots,N\).
\end{enumerate}
\end{theorem}

\begin{proof}
Assume, by contradiction, that there exists a solution \(x(t)\) with
\(x(0)\neq0\) and a first convergence time $t_s := \inf\{t > 0 : x(t) = 0\} < T$.

Since \(t_s<T\), one can choose \(\varepsilon>0\) sufficiently small
such that $I   :=  [t_s-\varepsilon,t_s+\varepsilon]\cap(0,T)$
is a compact interval strictly separated from \(T\). For each
\(k=1,\ldots,N\), the function
    $t\mapsto\Phi_k(T-t)$
is continuous and nonzero on \(I\). Therefore, by compactness,
\begin{equation*}
    M_k
    :=
    \max_{t\in I}
    \left|
        \frac{1}{\Phi_k(T-t)}
    \right|
    <\infty.
\end{equation*}

Because  \(P_0\) and \(P_k\) are locally Lipschitz continuous, there exists a
neighborhood \(B_r=\{x\in\mathbb{R}^n:\|x\|\leq r\}\) of the origin such that $\|P_0(t,x)-P_0(t,y)\|   \leq  L_0\|x-y\|$, $\|P_k(x)-P_k(y)\|   \leq  L_k\|x-y\|, x,y\in B_r$.
Consequently, for every \(t\in I\) and \(x,y\in B_r\),
\begin{align}
    \|f(t,x)-f(t,y)\|
    &\le L_0\|x-y\|+
    \left\|  
        \sum_{k=1}^{N}
        \frac{P_k(x)-P_k(y)}
             {\Phi_k(T-t)}
    \right\|
    \nonumber\\
    &\leq    L_0\|x-y\|+
    \sum_{k=1}^{N}M_kL_k\|x-y\|
    \nonumber\\
    &=
    L\|x-y\|, \label{lip}
\end{align}
where $L:=L_0+\sum_{k=1}^{N}M_kL_k<\infty$. Thus, \(f(t,x)\) is continuous in  \(t\in I\) and locally Lipschitz
continuous in \(x\).

By the Picard--Lindelöf local existence and uniqueness theorem
\cite[Theorem~3.1]{Khalil2002}, the initial-value
problem
\begin{equation}
    \dot{x}=f(t,x),
   \quad
    x(t_s)=0,
    \label{eq:ivp-at-ts}
\end{equation}
admits a unique local solution through the point \((t_s,0)\).

Since \(P_k(0)=0\) for every \(k\), one has
\[
    f(t,0)
    = P_0(t,0)+
    \sum_{k=1}^{N}
    \frac{P_k(0)}{\Phi_k(T-t)}
    =0,
   \quad t<T.
\]
Hence, $x(t)\equiv0$  
is also a solution of \eqref{eq:ivp-at-ts}. 

This contradicts the assumption that there exists a solution \(x(t)\) with
\(x(0)\neq 0\) and a first convergence time $t_s := \inf\{t > 0 : x(t) = 0\} < T$.
 Consequently, no nonzero solution can reach the origin at any time \(t_s<T\).
\end{proof}

\begin{remark}
 The first and second conditions require the 
complete closed-loop vector field to satisfy $f(t,0)=0$ and be locally Lipschitz in $x$ on every compact
subinterval of $[0,T)$. This condition is vital to avoid early convergence. The third condition is actually a normal time-varying gain of prescribed-time control.
In addition, to obtain a unique equilibrium of the system \eqref{mainsystem}, one can 
set $P_k(x)=0$ if and only if $x=\mathbf{0}$.
\end{remark}

\subsection{Local homogeneous linear-growth condition}

Although the Lipschitz condition of \(P(x)\) is vital in the proof of Theorem \ref{maintheorem}, we emphasize that the loss of local Lipschitz continuity of \(P(x)\) only invalidates the uniqueness-based non-reachability proof. The Lipschitz condition is not sufficient to guarantee convergence before the prescribed time. 
 The following example shows that a non-Lipschitz \(P(x)\) also ensures an exact convergence time.

\begin{example}
Consider the scalar system 
\begin{equation*}
\dot{x}(t)=-\dfrac{x(t)\left[2+\sin\left(1/x(t)\right)\right]}{T-t},
\end{equation*}
where \(T>0\) and \(0\le t<T\), $\dot{x}(t)=0$ when $x(t)=0$.
Denote the state-dependent part as 
\begin{equation*}
g(x)=x\left[2+\sin\left(1/x\right)\right].
\end{equation*}
For every $x\neq0$,
\begin{equation*}
g'(x)
=
2+\sin\left(\frac{1}{x}\right)
-\frac{1}{x}\cos\left(\frac{1}{x}\right).
\end{equation*}
Let $x_n=\frac{1}{2n\pi}$. As $x_n\rightarrow0$, we have
\begin{equation*}
g'(x_n)=2-2n\pi\rightarrow-\infty.
\end{equation*}
Thus, $g'$ is unbounded in every neighborhood of the origin.
Consequently, $g$ is not locally Lipschitz at $x=0$.

Since $1\leq 2+\sin\left(\frac{1}{x}\right)\leq3$, one has $|x|\leq|g(x)|\leq3|x|$.
For any $x(0)\ne 0$, it is not difficult to obtain that the sign of $x(t)$ does not change when $t<T$.
We have:
\begin{equation*}
-\frac{3}{T-t}|x|
\leq\dot{|x|}
\leq-\frac{1}{T-t}|x|.
\end{equation*}
Its solution satisfies
\begin{equation*}
|x_0|
\left(\frac{T-t}{T}\right)^3
\leq |x(t)|
\leq
|x_0|
\left(\frac{T-t}{T}\right).
\end{equation*}
Therefore, $\lim_{t\to T^-}|x(t)|=0$, and $x(t)\neq0, \forall t<T$,
 despite the loss of local Lipschitz condition at the origin.
\end{example}

\begin{remark}
Although $x\left[2+\sin\left(1/x\right)\right]$
is not locally Lipschitz at $x=0$, it satisfies the homogeneous
linear-growth condition (see Assumption \ref{linear-growth}).
This condition  excludes convergence before $T$. 
The relationship between Lipschitz continuity and homogeneous
linear growth can be summarized as follows. If $P(t,0)=0$ and
$P(t,x)$ is locally Lipschitz continuous in $x$, then, in a
neighborhood of the origin, $P(t,x)$ has a local homogeneous
linear-growth bound. However, the converse is generally false. Therefore, the
homogeneous linear-growth condition provides a strictly more
general sufficient condition to exclude premature convergence.
\end{remark}

\begin{assumption}
\label{linear-growth}
For the system \eqref{mainsystem},
there exist a radius $r>0$, a nonnegative locally integrable function
$a_0(t) \in L_{\mathrm{loc}}^{1}([0,T))$, and constants
$a_k\geq0$, $k=1,\ldots,N$, such that
\begin{equation*}
\|P_0(t,x)\|
\leq a_0(t)\|x\|, \quad
\|P_k(x)\|
\leq a_k\|x\|,
\end{equation*}
for all $\|x\|\leq r$ and $t<T$.
\end{assumption}

Since $ a\|x\|_1
\leq \sqrt{n}\,a\|x\|_2$, for the system \eqref{mainsystem} with finite dimension $n$, Assumption \ref{linear-growth} is equal to the linear-growth condition \cite{zhouassumption}.

\begin{theorem}
\label{maintheorem2}
Consider the system \eqref{mainsystem}.
Suppose that Assumption~\ref{linear-growth} holds and
  $\Phi_k(T-t):[0,T)\to\mathbb{R}_{> 0}$ is continuous.
Then every absolutely continuous solution of
\eqref{mainsystem} with $x(0)\neq0$ satisfies
$x(t)\neq0,
0\leq t<T$.
Consequently, no nonzero trajectory can reach the origin before the
prescribed time $T$.
\end{theorem}

\begin{proof}
Assume, by contradiction, that there exists a solution \(x(t)\) with
\(x(0)\neq0\) and a first convergence time $t_s := \inf\{t > 0 : x(t) = 0\} < T$.
Then $x(t)\neq 0, \forall 0\leq t<t_s $ and $x(t_s)=0$.

Since $x(t)$ is continuous and $x(t_s)=0$, there exists
$t_r\in[0,t_s)$ such that
$0<\|x(t)\|\leq r,
t\in[t_r,t_s)$.
As demonstrated in the proof of Theorem \ref{maintheorem},
denote
\begin{equation*}
    M_k
    :=
    \sup_{t\in[t_r,t_s)}
    \left|
        \frac{1}{\Phi_k(T-t)}
    \right|
    <\infty.
\end{equation*}

For every $\|x\|\leq r$, Assumption \ref{linear-growth} gives
\begin{align*}
\|f(t,x)\|
&\leq
\|P_0(t,x)\|
+
\sum_{k=1}^{N}
\frac{\|P_k(x)\|}{\Phi_k(T-t)}
\nonumber\\
&\leq
\left[
a_0(t)
+
\sum_{k=1}^{N} a_k M_k
\right]\|x\|.
\end{align*}

Consider the following Lyapunov function:
\begin{equation*}
    V(x)=\frac12  x^\top x =\frac12\|x\|^2.
\end{equation*}
For any absolutely continuous solution, and for almost every
$t\in[t_r,t_s)$, the derivative of the Lyapunov function satisfies
\begin{align*}
\dot V(t)
&=
x^\top(t)f(t,x(t))
\nonumber\\
&\geq
-\|x(t)\|\,\|f(t,x(t))\|
\nonumber\\
&\geq
-2\ell(t)V(t),
\end{align*}
where
\begin{equation*}
\ell(t)
:=
a_0(t)
+
\sum_{k=1}^{N}
\frac{a_k}{\Phi_k(T-t)}.
\end{equation*}
Hence, the Lyapunov function satisfies
\begin{equation*}
V(t)
\geq
V(t_r)
\exp\left(
-2\int_{t_r}^{t}\ell(\tau)\,d\tau
\right),
\quad
t\in[t_r,t_s).
\end{equation*}
Letting $t\to t_s^{-}$ and using the continuity of $V(x(t))$, we obtain
\begin{equation*}
V(t_s)
\geq
V(t_r)
\exp\left(
-2\int_{t_r}^{t_s}\ell(\tau)\,d\tau
\right)
>0.
\end{equation*}

However, $x(t_s)=0$ implies $V(t_s)=0$, which is a contradiction.
Therefore, $x(t)\neq0,
0\leq t<T$.
In other words, the trajectory cannot reach the origin before the prescribed time $T$.
\end{proof}

\begin{remark}
Assumption~\ref{linear-growth} can be satisfied by a broad
class of nonlinear functions that vanish at the origin.
For example, if
$P(x)= c x ^q,
 q\geq1$,
then $\|P(x)\| \leq cr^{q-1}\|x\|$, $\forall \|x\|\leq r$.
Other commonly encountered nonlinearities also satisfy this
condition in a neighborhood of the origin, e.g., 
$\tan x, \frac{x}{1+|x|}$.

\end{remark}

\subsection{Osgood condition}

The following example shows that a nonlinear growth term $P (x)$
also ensures an exact
convergence time.

\begin{example} \label{exa4}
     Consider the scalar system
\begin{equation*}
    \dot{x}
    =
    -\frac{P(x)}{T-t},
    \quad 0\le t<T, \quad 0 \le |x|<1,
    \label{eq:nonlip_example}
\end{equation*}
where $P(x)=-x\ln\frac{1}{|x|}$ when  $0<|x|<1$, and $P(0)=0$.
The function \(P(x)\) is continuous at the origin, since
\begin{equation*}
    \lim_{x\to0}x\ln\frac{1}{|x|}=0.
\end{equation*}
However, \(P(x)\) is not locally Lipschitz at \(x=0\), and is  not even a homogeneous
linear-growth function, since
\begin{equation*}
    \lim_{x\to0}\ln\frac{1}{|x|}=+\infty.
\end{equation*}
Notice that the sign of \(x(t)\) does not change. Let $v=\ln\frac{1}{|x|}$, it is not difficult to obtain $\frac{\dot v}{v}=\frac{1}{T-t}$. Solving the differential equation, we have
\begin{equation*}
    |x(t)|=
    |x(0)|^{\frac{T}{T-t}}, \quad  0<x(0)<1.
\end{equation*}
Hence, $\lim_{t\to T^-}|x(t)|=0$, and $x(t)\neq0, \forall t<T$.
\end{example}

\begin{definition} \label{osgooddef}
Let $r>0$. A continuous and nondecreasing function $\omega:[0,r]\rightarrow\mathbb{R}_{\geq0}$
is called an Osgood function at the origin if
\begin{enumerate}
    \item  $\omega(0)=0,
\quad
\omega(s)>0,\quad s\in(0,r]$.
\item $\int_{0^+}^{r}\frac{ds}{\omega(s)}
:=\lim_{\varepsilon\to0^+}
\int_{\varepsilon}^{r}\frac{ds}{\omega(s)}
=+\infty$.
\end{enumerate}
\end{definition}

According to  Definition \ref{osgooddef}, choose
$r=e^{-1}$ and define the function
\[
\omega(s)=
\begin{cases}
s\ln\dfrac{1}{s}, & s\in(0,r],\\[1mm]
0, & s=0.
\end{cases}
\]
Since $\lim_{s\to0^+}s\ln\frac{1}{s}=0$,
the function $\omega$ is continuous on $[0,r]$. Moreover, for
$s\in(0,r]$, $\omega'(s)=\ln(1/s)-1\geq 0$.
Furthermore,
\[
\int_{0^+}^{r}\frac{\mathrm{d}s}{\omega(s)}
=
\lim_{\varepsilon\to0^+}
\int_{\varepsilon}^{r}
\frac{\mathrm{d}s}{s\ln(1/s)}
=+\infty.
\]
Therefore, $\omega$ is an Osgood function at the origin. 
In Example \ref{exa4}, we have 
$\|P(x)\|\leq \omega(\|x\|), \forall \|x\|\leq r$.

\begin{assumption}
\label{osgood}
Consider system \eqref{mainsystem}.
There exist a radius $r>0$, an Osgood function
$\omega:[0,r]\rightarrow \mathbb{R}_{\geq 0}$,
a nonnegative function $a_0\in L_{\mathrm{loc}}^1([0,T))$,
and constants $a_k\geq0$, $k=1,\ldots,N$, such that the following conditions hold:
\begin{enumerate}
    \item $\|P_0(t,x)\|\leq a_0(t)\omega(\|x\|)$ holds for all $\|x\|\leq r$ and almost every
$t\in[0,T)$.
   \item  $\|P_k(x)\|\leq a_k\omega(\|x\|), \forall \|x\|\leq r$.
\end{enumerate}
\end{assumption}

\begin{theorem}
\label{maintheorem3}
Suppose that Assumption~\ref{osgood} holds, and 
  $\Phi_k(T-t):[0,T)\to \mathbb{R}_{> 0}$ is continuous.
Then every absolutely continuous solution of
\eqref{mainsystem} with $x(0)\neq 0$ satisfies
$x(t)\neq0,
0\leq t<T$.
Consequently, no nonzero trajectory can reach the origin before
the prescribed time $T$.
\end{theorem}

\begin{proof}
Assume, by contradiction, that there exists a solution \(x(t)\) with
\(x(0)\neq0\) and a first convergence time $t_s := \inf\{t > 0 : x(t) = 0\} < T$.
Then $x(t)\neq 0, \forall 0\leq t<t_s $ and $x(t_s)=0$.

For every absolutely continuous solution, there exists
$t_0<t_s$ such that $0<\|x(t)\|\leq r,
t\in[t_0,t_s).$
Define
\begin{equation*}
\ell(t)
:=
a_0(t)
+
\sum_{k=1}^{N}
\frac{a_k}{\Phi_k(T-t)}.
\end{equation*}
Since $\Phi_k(T-t)$ is continuous and strictly positive on the
compact interval $[t_0,t_s]$,  one has $\ell(t)\in L^1([t_0,t_s])$.

For $\|x\|\leq r$, Assumption~\ref{osgood} gives
\begin{align*}
\|\dot{x}(t)\|
&\leq
\|P_0(t,x)\|
+
\sum_{k=1}^{N}
\frac{\|P_k(x)\|}{\Phi_k(T-t)}
\nonumber\\
&\leq
\left[
a_0(t)
+
\sum_{k=1}^{N}
\frac{a_k}{\Phi_k(T-t)}
\right]\omega(\|x\|)
\nonumber\\
&=
\ell(t)\omega(\|x\|).
\label{eq:full-osgood-bound}
\end{align*}

Let $\rho(t):=\|x(t)\|=\sqrt{x^\top(t)x(t)}$.
For almost every $t\in[t_0,t_s)$, we have
\begin{align}
\dot{\rho}(t)
=\frac{x^\top(t)\dot{x}(t)}{\|x(t)\|}
\nonumber\geq
\frac{-\|x^\top(t)\|\|\dot{x}(t)\|}{\|x(t)\|}
\nonumber\geq
-\ell(t)\omega(\rho(t)).
\end{align}

Consequently, for almost every $t\in[t_0,t_s)$,
\begin{equation*}
-\frac{\dot{\rho}(t)}{\omega(\rho(t))}
\leq
\ell(t).
\end{equation*}
Integrating from $t_0$ to $t<t_s$ yields
\begin{equation}
\int_{\rho(t)}^{\rho(t_0)}
\frac{ds}{\omega(s)}
\leq
\int_{t_0}^{t}\ell(\tau)\,d\tau.
\label{eq:osgood-comparison}
\end{equation}

When $t\to t_s^-$, the right-hand side remains finite because
$\ell\in L^1([t_0,t_s])$. However, $\rho(t)\to\rho(t_s)=0$
and the Osgood condition implies
\[
\int_{\rho(t)}^{\rho(t_0)}
\frac{ds}{\omega(s)}
\to+\infty.
\]
This contradicts \eqref{eq:osgood-comparison}. Therefore, no
nonzero trajectory can reach the origin at any time $t_s<T$.
\end{proof}

\begin{remark}
Theorems \ref{maintheorem}-\ref{maintheorem3} provide a hierarchy of sufficient conditions for excluding premature convergence.
Indeed, local Lipschitz continuity together with ($f(t,0)=0$) implies a local homogeneous linear-growth bound, while the linear-growth condition is a particular Osgood condition obtained by choosing ($\omega(s)=s$). Therefore, among the three conditions considered above, the Osgood condition provides the most general sufficient criterion for excluding convergence before the prescribed time. However, its
application requires the construction of a continuous and
nondecreasing modulus $\omega$ that dominates the nonlinear terms,
together with the verification of the Osgood integral condition.
Consequently, the applicability of Theorem~3 may be less
straightforward in practice.
\end{remark}

\begin{remark}
The most commonly encountered non-Osgood moduli are sublinear
fractional-power functions. For example, let $ \omega(s)=c s^{\alpha},
 c>0, 0<\alpha<1$.
Then
\[
\begin{aligned}
    \int_{0^+}^{r}\frac{\mathrm{d}s}{\omega(s)}
= \frac{1}{c}\int_{0^+}^{r}s^{-\alpha}\,\mathrm{d}s  
= \frac{r^{1-\alpha}}{c(1-\alpha)}
    <+\infty.
\end{aligned}
\]
Therefore, a fractional-power modulus does not satisfy the Osgood
condition. However, the product of non-Osgood individual functions can satisfy a linear-growth bound, e.g.,
\[
\left|x_1^{\frac{2}{3}}x_2^{\frac{1}{3}}\right|
=\left|x_1^{\frac{1}{3}}x_1^{\frac{1}{3}}x_2^{\frac{1}{3}}\right|
\leq
\frac{|x_1|+|x_1|+|x_2|}{3}.
\]
Non-Osgood behavior is not restricted to fractional-power functions.
For example,
\[
    \omega_1(s)
    =
    s\left(\ln\frac{1}{s}\right)^p,
    \quad p>1.
\]
It is continuous, positive, and nondecreasing in a sufficiently small neighborhood of
the origin, and 
$\omega_1(0)=0$. However, for sufficiently small $r$
\[
\begin{aligned}
    \int_{0^+}^{r}\frac{\mathrm{d}s}{\omega_1(s)}
=
    \int_{\ln(1/r)}^{+\infty}\frac{\mathrm{d}u}{u^p}
=
    \frac{\left(\ln\frac{1}{r}\right)^{1-p}}{p-1}
    <+\infty.
\end{aligned}
\]
Hence, $\omega_1(s)$ is not an Osgood function. 
\end{remark}

Nevertheless, failure of the Osgood condition
alone does not necessarily imply that premature convergence occurs.

\begin{example} \label{ex4}
Consider the following system:
\[
    \dot{x}
    =
    -\frac{x}{T-t}
    +
    \frac{Jx}
    {(T-t)\|x\|^{1-\alpha}},
    \quad x\neq0,\quad 0\leq t<T,
\]
with $\dot{x}=0$ at $x=\mathbf{0}$,
and $J=
    \begin{bmatrix}
        0&-1\\
        1&0
    \end{bmatrix},
 0<\alpha<1.$

Let $r(t)=\|x(t)\|$. Since 
$x^\top Jx=0$, we have
\[
\begin{aligned}
    \dot r
  =
    \frac{x^\top\dot x}{\|x\|}
=
    -\frac{r}{T-t}.
\end{aligned}
\]
Its solution is 
$ r(t)  =   r(0)(T-t)/T$. Therefore, for every nonzero initial condition, we have $x(t)\neq0, \forall t\in [0,T) $, and
$\lim_{t\to T^-}x(t)=0.$
Hence, the convergence time is exactly $T$.

On the other hand, since $x$ and $Jx$ are orthogonal,
\[
    \|\dot{x}\|
    =
    \frac{
        \sqrt{\|x\|^2+\|x\|^{2\alpha}}
    }{T-t}
    \geq
    \frac{\|x\|^\alpha}{T-t}.
\]
Although $\|x\|^\alpha$ cannot be limited by the Osgood function,
the system does not reach the origin before $T$. Therefore, the
Osgood condition is sufficient but not necessary for excluding
premature convergence in multidimensional systems.
\end{example}

\begin{remark}
Although the vector field in Example \ref{ex4} contains a non-Osgood term, this term is orthogonal to $x$ and affects only the angular motion. The radial dynamics remain $\dot r=-r/(T-t)$,
so that $r(t)>0$, $\forall t\in [0,T)$,  and $\lim_{t\to T^-}r(t)=0$. A general necessary and sufficient condition would need to characterize the exact radial component of the vector field along all trajectories, which is typically as difficult as directly analyzing the original system. In this paper, necessary and sufficient conditions that exclude
premature convergence are not considered further.
\end{remark}

\subsection{Extension to multi-agent systems}

Consider $N$ agents with states $x_i\in\mathbb{R}^n$, and define the
stacked state $X=\operatorname{col}(x_1,\ldots,x_N)\in\mathbb{R}^{Nn}.$
In multi-agent systems, the main objective is that the disagreement vector converges to the origin.
Suppose, for example, that the communication topology is described by a fixed undirected connected
graph with Laplacian matrix
$\mathcal{L}\in\mathbb{R}^{N\times N}$.
The Laplacian disagreement vector is defined by
\[
e=(\mathcal{L}\otimes I_n)X
=
\operatorname{col}\left(
\sum_{j=1}^{N}a_{ij}(x_i-x_j)
\right)_{i=1}^{N}.
\]

Suppose that the closed-loop dynamics induce the following closed
disagreement system:
\begin{equation}
\dot e
=
P_0(t,e)
+
\sum_{k=1}^{M}
\frac{P_k(e)}{\Phi_k(T-t)},
\quad 0\leq t<T,
\label{eq:mas_error_system}
\end{equation}
where each $\Phi_k(T-t)$ is continuous and strictly positive for
$t<T$. The set $e=0$, equivalently the consensus manifold
$\mathcal{C}$, is assumed to be positively invariant.

\begin{theorem} \label{theo4}
Consider the disagreement system \eqref{eq:mas_error_system}.  $\Phi_k(T-t):[0,T)\to \mathbb{R}_{> 0}$ is continuous.
Suppose that there exist $r>0$, an Osgood function
$\omega:[0,r]\to \mathbb{R}_{\geq 0}$, a nonnegative function
$a_0\in L_{\mathrm{loc}}^1([0,T))$, and constants $a_k\geq0$ such that the following conditions hold for all $\|e\|\leq r$:
\[
\|P_0(t,e)\|
\leq
a_0(t)\omega(\|e\|),
\]
\[
\|P_k(e)\|
\leq
a_k\omega(\|e\|),
\quad k=1,\ldots,M.
\]
 Then every absolutely continuous solution with $e(0)\neq0$ satisfies $e(t)\neq0$ for all $0\leq t<T$.
Consequently, no initially non-consensual trajectory can reach the
consensus manifold before the prescribed time $T$.
\end{theorem}

\begin{proof}
Suppose, by contradiction, that a solution with $e(0)\neq0$ reaches
consensus for the first time at
\[
t_s:=\inf\{t>0:e(t)=0\}<T.
\]
Then
\[
e(t)\neq0,\quad 0\leq t<t_s,
\quad
e(t_s)=0.
\]
By continuity, there exists $t_r<t_s$ such that
\[
0<\|e(t)\|\leq r,\quad t\in[t_r,t_s).
\]
Define
\[
\rho(t):=\|e(t)\|,
\quad
\ell(t):=
a_0(t)+
\sum_{k=1}^{M}\frac{a_k}{\Phi_k(T-t)}.
\]
Since $t_s<T$ and each $\Phi_k(T-t)$ is continuous and strictly
positive, one has
\[
\ell\in L^1([t_r,t_s]).
\]

For almost every $t\in[t_r,t_s)$, the assumed growth bounds yield
\[
\|\dot e(t)\|
\leq
\|P_0(t,e(t))\|
+
\sum_{k=1}^{M}
\frac{\|P_k(e(t))\|}{\Phi_k(T-t)}
\leq
\ell(t)\omega(\rho(t)).
\]
Since $\rho(t)>0$ on $[t_r,t_s)$, it follows for almost every
$t\in[t_r,t_s)$ that
\[
\dot\rho(t)
=
\frac{e^\top(t)\dot e(t)}{\|e(t)\|}
\geq
-\|\dot e(t)\|
\geq
-\ell(t)\omega(\rho(t)).
\]
Integrating this differential inequality gives
\begin{equation}
\int_{\rho(t)}^{\rho(t_r)}
\frac{\mathrm{d}s}{\omega(s)}
\leq
\int_{t_r}^{t}\ell(\tau)\,\mathrm{d}\tau,
\quad t<t_s.
\label{eq:mas_osgood_integral}
\end{equation}

As $t\to t_s^-$, the right-hand side of
\eqref{eq:mas_osgood_integral} remains finite because
$\ell\in L^1([t_r,t_s])$. However, $\rho(t)\to0$, and the Osgood
property of $\omega$ implies
\[
\int_{\rho(t)}^{\rho(t_r)}
\frac{\mathrm{d}s}{\omega(s)}
\longrightarrow+\infty.
\]
This contradiction proves that $e(t)\neq0$ for every $t<T$.
\end{proof}

\begin{remark}
Choosing $\omega(s)=s$ reduces the Osgood condition to the local
homogeneous linear-growth condition
\[
\|P_0(t,e)\|
\leq
a_0(t)\|e\|,
\quad
\|P_k(e)\|
\leq
a_k\|e\|,
\quad k=1,\ldots,M.
\]
Moreover, suppose that $P_0(t,0)=0$ and $P_k(0)=0$, and that the error
vector fields are locally Lipschitz at $e=0$, i.e.,
\[
\|P_0(t,e)-P_0(t,\bar e)\|
\leq
L_0(t)\|e-\bar e\|,
\quad
L_0\in L_{\mathrm{loc}}^1([0,T)),
\]
\[
\|P_k(e)-P_k(\bar e)\|
\leq
L_k\|e-\bar e\|,
\quad k=1,\ldots,M.
\]
Setting $\bar e=0$ gives the corresponding local homogeneous
linear-growth bounds. Therefore, local Lipschitz continuity, local
homogeneous linear growth, and the Osgood condition form a hierarchy
of sufficient conditions for excluding premature consensus.
\end{remark}

\section{Prescribed-time control systems analysis}
In this section, several prescribed-time control systems are shown to converge to zero exactly at the prescribed time using the proposed theorems, demonstrating a high degree of generality.

\subsection{Prescribed-time stabilization with backstepping}

In this subsection, we apply Theorem~\ref{maintheorem} to a backstepping-based prescribed-time control system \cite{duibiao}. 

Consider the following $n$th-order system:
\begin{equation}  	
	\begin{cases} 
		\dot x_j=x_{j+1}+\theta^\top f_j(\bar{x}_j),\, j=1,2,...,n-1,\\
		\dot x_n=u+\theta^\top f_n(\bar{x}_n).
	\end{cases} \label{scalar}
\end{equation}
where $x_j\in\mathbb{R}$, $j=1,\ldots,n$, $u\in\mathbb{R}$,
and $\theta\in\mathbb{R}^r$ are the state components, control input,
and uncertain parameter vector, respectively.
$\bar{x}_j
    =
    \begin{bmatrix}
        x_1, & \cdots & ,x_j
    \end{bmatrix}^{\mathrm T}
    \in\mathbb{R}^j,
    x(t)=\bar{x}_n(t)$.
The mappings
$f_j:\mathbb{R}^j\rightarrow\mathbb{R}^r$
are known smooth functions satisfying $f_j(0)=0,j=1,\ldots,n$.

The controller in \cite{duibiao} is designed by backstepping. Consider the following transformation:
\begin{equation}  	
	\begin{cases}z_1 = x_1,\\
		z_j = x_j-\alpha_{j-1},\, j=2,3,...,n,
	\end{cases} \label{zsystems}
\end{equation}
where $\alpha_j$ is the virtual control to be designed.

The virtual control $\alpha_1$ is designed as:
\begin{equation*} \label{alpha1s}
	\alpha_1=- {\hat{\theta}}^\top f_1  -\frac{k_1 z_1}{T-t},
\end{equation*}
where $k_j>n-j+1$, $\forall j=1,...,n$, $\hat{\theta}$ denotes the estimate of $\theta$. Its update law is designed in the last step.

Let $\tau_1=z_1 f_1$ and $\tau_j=\tau_{j-1}+z_j \big(f_j-\sum_{\iota=1}^{j-1}\frac{\partial {\alpha _{j-1}}}{\partial {x_\iota}} f_\iota \big)$.
The virtual control $\alpha_2$ is
\begin{equation*} 
	\begin{split}
		\alpha_2=& -x_1 -\frac{k_2 x_2}{T-t}+\frac{k_2  \alpha_1}{T-t}	-{\hat{\theta}}^\top f_2 \\
		& +\left[\frac{\partial {\alpha _1}}{\partial {t}}+   \frac{\partial {\alpha _1}}{\partial {{\hat{\theta}}^\top}} \tau_2 + \frac{\partial {\alpha _1}}{\partial {x_1}} (x_2+\hat{ \theta}^\top f_1 ) \right].
	\end{split} 
\end{equation*}

The virtual control $\alpha_j$ is designed as:
\begin{eqnarray*}
	\alpha_j&=& \alpha_{j-2} -x_{j-1}-\frac{k_j  x_j}{T-t}+\frac{k_j \alpha_{j-1}}{T-t}	-{\hat{\theta}}^\top f_j \notag\\
	&&+ \left[\frac{\partial {\alpha _{j-1}}}{\partial {t}}+  \frac{\partial {\alpha _{j-1}}}{\partial {{\hat{\theta}}^\top}} \tau_j + \sum_{\iota=1}^{j-1}  \frac{\partial {\alpha _{j-1}}}{\partial {x_\iota}} (x_{\iota+1}+  {\hat{\theta}}^\top f_\iota ) \right]  \notag\\
	&&+   \left(\sum_{\iota=2}^{j-1} \frac{\partial {\alpha_{\iota-1}}}{\partial {{\hat{\theta}}^\top}} z_\iota \right)  \left(f_j-\sum_{\iota=1}^{j-1} \frac{\partial {\alpha _{j-1}}}{\partial {x_\iota}} f_\iota \right)  .
%	\label{alphajs}
\end{eqnarray*}

The control   input $u$ is designed as:
\begin{eqnarray} 
	u&=& \alpha_{n-2} -x_{n-1}-\frac{k_n  x_n}{T-t}+\frac{k_n \alpha_{n-1}}{T-t}	-{\hat{\theta}}^\top f_n \notag\\
	&&+ \left[\frac{\partial {\alpha _{n-1}}}{\partial {t}}+  \frac{\partial {\alpha _{n-1}}}{\partial {{\hat{\theta}}^\top}} \tau_n + \sum_{\iota=1}^{n-1} \frac{\partial {\alpha _{n-1}}}{\partial {x_\iota}} (x_{\iota+1}+  {\hat{\theta}}^\top f_\iota ) \right] \notag\\
	&&+  \left(\sum_{\iota=2}^{n-1} \frac{\partial {\alpha_{\iota-1}}}{\partial {{\hat{\theta}}^\top}} z_\iota \right)   \left(f_n-\sum_{\iota=1}^{n-1} \frac{\partial {\alpha _{n-1}}}{\partial {x_\iota}} f_\iota \right).
	%\label{controlvs}
\end{eqnarray}

Moreover, the update law of $\hat{\theta}$ is
\begin{equation} 
	\dot{\hat{\theta}}= \tau_n=\tau_{n-1}+z_n \big(f_n-\sum_{\iota=1}^{n-1}\frac{\partial {\alpha _{n-1}}}{\partial {x_\iota}} f_\iota \big), \, \hat{\theta}(t=0)=0. 
\end{equation}

For $t  \ge T$, let $u=0,	\dot{\hat{\theta}}=0$.
\begin{corollary}
Consider system~\eqref{scalar} under the 
backstepping controller and adaptive law given above. Suppose that
$f_j$, $j=1,\ldots,n$, are smooth and satisfy $f_j(0)=0$. Then, for
every initial condition $x(0)\neq 0$, the corresponding closed-loop
trajectory satisfies $x(t)\neq 0, \forall\,t\in[0,T)$.
\end{corollary}

\begin{proof}
Let $\vartheta\in\mathbb{R}^{r}$ be an arbitrary fixed vector
and define the shifted augmented state $\xi_{\vartheta}
    :=
    \operatorname{col}
    \bigl(x,\hat{\theta}-\vartheta\bigr).$
    
All explicit time singularities in the backstepping controller
are generated by $(T-t)^{-1}$. Since the controller is
constructed through a finite recursion, the augmented
closed-loop system can be written as
\begin{equation}
    \dot{\xi}_{\vartheta}
    =
    P_{0}(t,\xi_{\vartheta};\vartheta)
    +
    \sum_{\ell=1}^{N}
    \frac{
        P_{\ell}(\xi_{\vartheta};\vartheta)
    }{
        \Phi_{\ell}(T-t)
    },
    \quad 0\leq t<T,
    \label{eq:backstepping-general-form}
\end{equation}
where $ \Phi_{\ell}(T-t)=(T-t)^{q_{\ell}},
   q_{\ell}\in\mathbb{N}.$
Hence, every $\Phi_{\ell}(T-t)$ is continuous and strictly
positive for $t<T$.

It remains to verify the conditions on the numerator mappings.
At the first step,
\[
    \alpha_{1}
    =
    -\hat{\theta}^{\top}f_{1}
    -
    \frac{k_{1}z_{1}}{T-t},
    \quad
    \tau_{1}=z_{1}f_{1}.
\]
Since $f_{1}$ is smooth and $f_{1}(0)=0$, all corresponding
numerators are locally Lipschitz in $\xi_{\vartheta}$ and
vanish at $\xi_{\vartheta}=0$.

Suppose that this property holds up to the $(j-1)$th step.
The quantities
\[
    z_j=x_j-\alpha_{j-1},
\]
\[
    \tau_j
    =
    \tau_{j-1}
    +
    z_j
    \left(
        f_j-
        \sum_{\iota=1}^{j-1}
        \frac{\partial\alpha_{j-1}}{\partial x_\iota}
        f_\iota
    \right),
\]
and $\alpha_j$ are recursively constructed through finite
sums, products and partial derivatives of smooth functions,
together with powers of $(T-t)^{-1}$. Therefore, all newly
generated numerator mappings remain locally Lipschitz.

Moreover, at $\xi_{\vartheta}=0$, one has $ x=\mathbf{0},\hat{\theta}=\vartheta.$
Using $f_j(0)=0$ and the above recursive definitions gives
\[
    z_j=0,\quad
    \alpha_j=0,\quad
    \tau_j=0,
    \quad j=1,\ldots,n.
\]
In particular,
\[
    u(t,0;\vartheta)=0,
    \quad
    \dot{\hat{\theta}}(t,0;\vartheta)=\tau_n(t,0;\vartheta)=0.
\]
Thus,
\[
    P_0(t,0;\vartheta)=0,
    \quad
    P_\ell(0;\vartheta)=0,
    \quad \ell=1,\ldots,N,
\]
and all numerator mappings in
\eqref{eq:backstepping-general-form} are locally Lipschitz.
Therefore, for every fixed $\vartheta$, all the conditions of
Theorem~\ref{maintheorem} are satisfied.

Now let $t_c\in[0,T)$ be arbitrary and choose the fixed vector $\vartheta_c:=\hat{\theta}(t_c).$
Since $x(0)\neq0$,
\[
    \xi_{\vartheta_c}(0)
    =
    \operatorname{col}
    \bigl(x(0),\hat{\theta}(0)-\vartheta_c\bigr)
    \neq0.
\]
By Theorem~\ref{maintheorem}, $\xi_{\vartheta_c}(t_c)\neq0.$
However,
\[
    \xi_{\vartheta_c}(t_c)
    =
    \operatorname{col}
    \bigl(x(t_c),
    \hat{\theta}(t_c)-\vartheta_c\bigr)
    =
    \operatorname{col}\bigl(x(t_c),0\bigr).
\]
Hence, $ x(t_c)\neq0.$ Moreover,
since $t_c<T$ is arbitrary, $x(t)\neq0,\forall\,t\in[0,T).$
\end{proof}

In addition, Theorem~2 of \cite{duibiao} has already proven that $\lim_{t\to T^-}x(t)=0$.
 Therefore, the  closed-loop system is exact $T$-finite-time stable.

\subsection{Prescribed-time stabilization with PLE}

In this subsection, we apply Theorem~\ref{maintheorem2} to the parametric Lyapunov equations (PLEs)-based prescribed-time control systems \cite{zhouassumption,Zhang2025PrescribedTime}.

Consider the following system with $t\in[0,T)$,
\begin{equation*}
    \dot{x}=Ax+bu+\phi(t,x,u),
    \label{eq:zhou_plant}
\end{equation*}
where $x\in\mathbb{R}^n$, $u\in\mathbb{R}$, and
\begin{equation*}
A=
\begin{bmatrix}
0&1&0&\cdots&0\\
0&0&1&\cdots&0\\
\vdots&&\ddots&\ddots&\vdots\\
0&0&\cdots&0&1\\
0&0&\cdots&0&0
\end{bmatrix},
\quad
b=
\begin{bmatrix}
0\\
\vdots\\
0\\
1
\end{bmatrix}.
\end{equation*}
The nonlinear term is written as
\begin{equation*}
    \phi(t,x,u)
    =
    \begin{bmatrix}
        \phi_1(t,x,u)&\cdots&\phi_n(t,x,u)
    \end{bmatrix}^\top .
\end{equation*}
The nonlinearities satisfy a triangular linear-growth condition of the form
\begin{equation} \label{philine}
    |\phi_i(t,x,u)|
    \le
    \sum_{j=1}^{i}c_{ij}|x_j|  + \sum_{j=1}^{i}d_{ij}|u|, i=1,\ldots,n,
\end{equation}
where $c_{ij},d_{ij}\ge0$ and
    $\phi(t,0,u)=0$.

Let $P(\gamma)$ be the 
positive-definite solution of the parametric Lyapunov equation
\begin{equation*}
    A^\top P(\gamma)
    +P(\gamma)A
    -P(\gamma)bb^\top P(\gamma)
    =
    -\gamma P(\gamma),
    \label{eq:ple}
\end{equation*}
where the function $\gamma:[0,T)\to
\mathbb{R}_{>0}$ is continuous. The prescribed-time state-feedback controller is
\begin{equation*}
    u(t,x)
    =
    -\frac{1}{2}b^\top P(\gamma(t))x,\quad 0\le t<T,
\end{equation*}
where  $ 0<\gamma(t)<+\infty$, $ t<T$,
and  $\lim_{t\to T^-}\gamma(t)=+\infty$ .
Therefore, the closed-loop system is 
\begin{equation}
     \dot{x}=
     \left[
        A-\frac{1}{2}bb^\top P(\gamma(t))
    \right]x
    +
    \phi\bigl(t,x,-\frac{1}{2}b^\top P(\gamma(t))x\bigr).
    \label{eq:G}
\end{equation}

\begin{corollary}
Consider the system~\eqref{eq:G} under the 
PLE controller given above. 
Suppose the condition \eqref{philine} is satisfied.
 Then, for all $x(0)\neq 0$, the corresponding closed-loop
trajectory satisfies $x(t)\neq 0, \forall\,t\in[0,T)$.
\end{corollary}

\begin{proof}
The PLE solution \cite{zhouassumption} has the scaling structure
\begin{equation*}
    P(\gamma)
    =
    \gamma D_n(\gamma)P_nD_n(\gamma),
\end{equation*}
where $D_n(\gamma)
    =
    \operatorname{diag}
    \left\{
        \gamma^{n-1},
        \gamma^{n-2},
        \ldots,
        \gamma,
        1
    \right\},
    P_n=P(1).$

Let $p_{ij}$ denote the $(i,j)$th entry of $P_n$. Since $b=e_n$, the time-varying feedback term can be expanded as
\begin{equation*}
    -\frac{1}{2}bb^\top P(\gamma(t))x
    =
    -\frac{1}{2}
    \sum_{j=1}^{n}
    p_{nj}\gamma^{\,n-j+1}(t)e_nx_j.
\end{equation*}

For $k=1,\ldots,n$, define
\begin{equation*}
    H_k
    :=
    -\frac{1}{2}
    p_{n,n-k+1}
    e_ne_{n-k+1}^\top .
\end{equation*}
Then
\begin{equation*}
    -\frac{1}{2}bb^\top P(\gamma(t))x
    =
    \sum_{k=1}^{n}\gamma^k(t)H_kx.
\end{equation*}

Let $\Phi_k(T-t):=\frac{1}{\gamma^k(t)}$. Then we have
   $ \lim_{t\to T^-}\Phi_k(T-t)=0$.
According to  inequality \eqref{philine},
the closed-loop system \eqref{eq:G} satisfies all the conditions of Theorem \ref{maintheorem2}.
Therefore, no nonzero trajectory of the closed-loop system \eqref{eq:G} can
reach the origin before the prescribed time $T$.
\end{proof}

In addition, Theorem~1 of \cite{zhouassumption} has already proven that $\lim_{t\to T^-}x(t)=0$.
Hence, the controller of  \cite{zhouassumption} achieves exact $T$-finite-time stability in Definition \ref{exactT}.

\subsection{Prescribed-Time Consensus}

In this subsection, Theorem~\ref{theo4} is applied to the
distributed prescribed-time leader-following controller proposed in
\cite{Ye2025Decentralized}.

Consider the large-scale interconnected multiagent system consisting
of $N$ high-order nonlinear subsystems
\[
\dot{x}_i
 =
 Ax_i+Bu_i+f_i(t,x_1,\ldots,x_N),
\quad
y_i=Cx_i,
\]
where $x_i\in\mathbb{R}^n$, $u_i\in\mathbb{R}$, and only the output
$y_i$ is available for feedback. The nonlinear term $f_i$ contains
the uncertain local dynamics and the physical interactions among the
subsystems.

The communicating agents are divided into several leader--follower
groups. In the $h$-th group, the agent indexed by $c_h$ is the leader,
whereas the remaining agents are followers. The follower graph is
undirected and connected, and every follower can receive information
from its corresponding leader. Consider the distributed case
$\Omega_B=\varnothing$ and suppose that Assumption~2 and the
communication conditions in \cite{Ye2025Decentralized} hold, with
bounded leader inputs.

Recall the time-varying gain and scaling matrix
\[
r(t)=\frac{b_2}{T-t},
\quad
\Lambda_r=\operatorname{diag}\{r^n,r^{n-1},\ldots,r\}.
\]
For a follower $i$ in the $h$th group, define
\[
\eta_i=\Lambda_r(\hat{x}_i-\hat{x}_{c_h}),
\quad
\varepsilon_i
 =
 \Lambda_r(x_i-x_{c_h})-\eta_i.
\]
Stacking the scaled errors of all followers gives
\[
\xi=\operatorname{col}(\eta_i),
\quad
\varepsilon=\operatorname{col}(\varepsilon_i),
\quad
\chi=\operatorname{col}(\xi,\varepsilon).
\]
Since $\Lambda_r(t)$ is nonsingular for every $t<T$, $\chi=0$
implies
\[
x_i=x_{c_h},
\quad
\hat{x}_i=\hat{x}_{c_h}
\]
for every follower in the $h$th group.

\begin{corollary}
\label{cor:exact-consensus}
Consider the interconnected multiagent system above under the
distributed observer and controller proposed in
\cite{Ye2025Decentralized}. Suppose that $\Omega_B=\varnothing$,
Assumption~2 and the communication conditions in
\cite{Ye2025Decentralized} hold, and all leader inputs are bounded.
Then every solution satisfying $\chi(0)\neq0$ satisfies $\chi(t)\neq0,
0\leq t<T.$
\end{corollary}

\begin{proof}
Under the transformed time $ \tau=\ln(T/(T-t))$,
the closed-loop disagreement dynamics derived in
\cite{Ye2025Decentralized} are
\[
\begin{aligned}
\frac{d\xi}{d\tau}
&=
b_2\bigl(\mathcal{A}_c\xi
+\mathcal{A}_{\phi}\varepsilon\bigr)
+\mathcal{D}_0\xi,\\
\frac{d\varepsilon}{d\tau}
&=
b_2\bigl(\mathcal{A}_0\varepsilon
+\bar{F}(t,\chi)\bigr)
+\mathcal{D}_0\varepsilon,
\end{aligned}
\]
where $\bar{F}$ collects the nonlinear differences between the
followers and their respective leaders. Since
$d\tau/dt=(T-t)^{-1}$, the corresponding dynamics on the original
time axis can be written as
\begin{equation}
\dot{\chi}
 =
 P_0(t,\chi)
 +
 \frac{P_1(\chi)}{T-t},
\quad
0\leq t<T,
\label{eq:closed-consensus-error}
\end{equation}
where
\[
P_0(t,\chi)
 =
 \frac{b_2}{T-t}
 \operatorname{col}\bigl(0,\bar{F}(t,\chi)\bigr),
\]
\[
P_1(\chi)
 =
 \mathcal{A}_{\chi}\chi,
\quad
\mathcal{A}_{\chi}
 =
 \begin{bmatrix}
 b_2\mathcal{A}_c+\mathcal{D}_0
 &
 b_2\mathcal{A}_{\phi}
 \\
 0
 &
 b_2\mathcal{A}_0+\mathcal{D}_0
 \end{bmatrix}.
\]

If $\chi=0$, all follower--leader state and observer differences
vanish. Assumption~2 of \cite{Ye2025Decentralized} then gives
$\bar{F}(t,0)=0$. Hence,
\[
P_0(t,0)=0,
\quad
P_1(0)=0,
\]
and the set $\{\chi=0\}$ is positively invariant.

Furthermore, Assumption~2 and the scaled coordinates yield, for some
constant $c_F>0$,
\[
\|\bar{F}(t,\chi)\|
 \leq
 c_F\frac{\theta(t)}{r(t)}\|\chi\|.
\]
Using $r(t)=b_2/(T-t)$ gives
\[
\|P_0(t,\chi)\|
 \leq
 c_F\theta(t)\|\chi\|,
\quad
\|P_1(\chi)\|
 \leq
 \|\mathcal{A}_{\chi}\|\|\chi\|.
\]
Since $\theta(t)$ is locally integrable on $[0,T)$, the growth
conditions in Theorem~\ref{theo4} hold with
\[
a_0(t)=c_F\theta(t),
\quad
a_1=\|\mathcal{A}_{\chi}\|,
\quad
\Phi_1(T-t)=T-t.
\]

Choose $\omega(s)=s$ on interval $[0,r]$. Then $\omega$ is an Osgood function because
\[
\int_{0^+}^{r}\frac{ds}{\omega(s)}
 =
 \int_{0^+}^{r}\frac{ds}{s}
 =+\infty.
\]
Therefore, Theorem~\ref{theo4} implies
\[
\chi(0)\neq0
\Longrightarrow
\chi(t)\neq0, \quad
0\leq t<T.
\] 
The consensus cannot be established before  $t=T$.
\end{proof}

On the other hand, the prescribed-time stability result in
\cite{Ye2025Decentralized} establishes $\lim_{t\to T^-}\xi(t)=0,
\lim_{t\to T^-}\varepsilon(t)=0$,
and hence $\lim_{t\to T^-}\chi(t)=0$.
Therefore, the consensus is achieved exactly at $t=T$.

\begin{remark}
 The control methods considered in this section are only representative examples. They do not establish a one-to-one correspondence between the proposed sufficient conditions and specific controller design methods. The applicability of the proposed framework mainly depends on the regularity and local growth properties of the resulting closed-loop vector field. Therefore, the proposed results are not limited to adaptive backstepping, PLEs-based control, or distributed leader-following consensus. Other prescribed-time controller design methods can also use these results to exclude premature convergence and verify exact convergence at the prescribed time, provided that the corresponding closed-loop system satisfies one of the sufficient conditions developed in this paper. This further demonstrates the generality of the proposed framework.
\end{remark}

\section{Conclusions}
\label{section6}
This paper has established a unified theoretical framework for excluding premature convergence in prescribed-time control systems. It has been shown that local Lipschitz continuity, a local homogeneous linear-growth bound, and the Osgood condition each provide a sufficient criterion ensuring that every trajectory starting from a nonzero initial condition remains nonzero before the prescribed time. Specifically, local Lipschitz continuity, together with the equilibrium condition, implies local homogeneous linear growth, while the linear-growth condition is a particular case of the Osgood condition. Therefore, the Osgood condition offers the most general criterion, although its verification may be less straightforward. When combined with an established prescribed-time stabilization, the exclusion of premature convergence guarantees that the settling time is exactly the prescribed time instant $T$. 
Applications to adaptive backstepping control, parametric-Lyapunov-equation-based stabilization, and distributed leader-following consensus further demonstrate the generality of the proposed results. These analyses provide a reusable tool for upgrading existing prescribed-time convergence results into rigorous exact-settling-time guarantees.

\bibliographystyle{IEEEtran}
\bibliography{references}

%\ifincludeauthorbios
%\input{sections/06_biographies}
%\fi

\end{document}